\documentclass[a4paper,twoside,11pt,reqno]{amsart}

\usepackage[hmargin={20mm,20mm},vmargin={20mm,20mm}]{geometry}
\usepackage{enumerate}
\usepackage[T1]{fontenc}
\numberwithin{equation}{section}

\newcommand{\E}{\mathbb{E}\,}
\newtheorem{theorem}{Theorem}[section]

\newtheorem{proposition}[theorem]{Proposition}
\newtheorem{remark}[theorem]{Remark}
\newtheorem{example}[theorem]{Example}
\newtheorem{definition}[theorem]{Definition}

\newcommand{\N}{\mathbb{N}}

\newcommand{\lcx}{\leq_{\text{\rm cx}}}

\newcommand{\C}{\mathcal{C}}

\begin{document}


\title{A sharpening of a problem on Bernstein polynomials and convex function and related results}

\author{Andrzej Komisarski}
\email{andkom@math.uni.lodz.pl}
\address{Department of Probability Theory and Statistics, Faculty of Mathematics and Computer Science,
University of \L\'od\'z, ul. Banacha 22, 90-238 \L\'od\'z, Poland}
\author{Teresa Rajba}
\email{trajba@ath.bielsko.pl}
\address{Department of Mathematics, University of Bielsko-Bia\l a, ul. Willowa 2, 43-309 Bielsko-Bia\l a, Poland}

\begin{abstract}
We present a short proof of a conjecture proposed by I.\ Ra\c{s}a (2017), which is an inequality involving basic Bernstein polynomials and convex functions. This proof was given in the letter to I.\ Ra\c{s}a (2017). The  methods of our proof allow us to obtain some extended versions of this inequality as well as other inequalities given by I.\ Ra\c{s}a.
As a tool we use stochastic convex ordering relations.
We propose also some generalizations of the binomial convex concentration inequality. We use it to insert some additional expressions between left and right sides of the Ra\c{s}a inequalities.
\end{abstract}





\maketitle

\section{Introduction}
For $n\in\N$ 
the classical Bernstein operators $B_n:\C([0,1])\to\C([0,1])$, defined by
$$B_n(f)(x)=\sum_{i=0}^n p_{n,i}(x)f\left(\tfrac in\right)\quad\text{for } x\in[0,1],$$\\
with the Bernstein basic polynomials 
\[
 p_{n,i}(x)=\binom{n}{i}x^i(1-x)^{n-i} \quad \text{for } \ i=0,1,\dots,n, \  x\in[0,1],
\]
are the most prominent positive linear approximation operators (see \cite{Lorentz1953}).
If $f\in\C([0,1])$ is a convex function, the inequality 
\begin{equation}\label{eq:Rasa}
 \sum_{i=0}^n \sum_{j=0}^n \left(p_{n,i}(x)p_{n,j}(x)+p_{n,i}(y)p_{n,j}(y)-2p_{n,i}(x)p_{n,j}(y)\right)f\left(\frac{i+j}{2n}\right)\geq0
\end{equation}
is valid for all $x,y\in[0,1]$.

This inequality involving Bernstein basic polynomials and convex functions was stated as an open problems 25 years ago by I.\ Ra\c{s}a. During the Conference on Ulam's Type Stability (Rytro, Poland, 2014), Ra\c{s}a \cite{Rasa2014b} recalled his problem.

Inequalities of type \eqref{eq:Rasa} have important applications. They are useful when studying whether the Bernstein-Schnabl operators preserve convexity (see \cite{Rasa2014, Rasa2017}).

Recently, J.\ Mrowiec, T.\ Rajba and S.\ W\k{a}sowicz \cite{MRW2017} affirmed the conjecture \eqref{eq:Rasa} in positive. Their proof makes heavy use of probability theory.
As a tool they applied a concept of stochastic convex orderings, as well as the so-called binomial convex concentration inequality. Later,
U.\ Abel \cite{Abel2016} gave an elementary proof of \eqref{eq:Rasa}, which was much shorter than that given in \cite{MRW2017}. Very recently, A. Komisarski and T. Rajba \cite{KomRaj2018} gave a new, very short proof of \eqref{eq:Rasa}, which is significantly simpler and shorter than that given by U.\ Abel \cite{Abel2016}. As a tool the authors use both stochastic convex orders as well as the usual stochastic order.
 
Let us recall some basic notations and results on stochastic ordering (see \cite{Shaked2007}).
If $\mu$ and $\nu$ are two probability distributions such that
$$\int\varphi(x)\mu(dx)\leq\int\varphi(x)\nu(dx) \quad \text{for all convex functions }\ \varphi\colon\mathbb R\to\mathbb R,$$
provided the integrals exist, then $\mu$ is said to be \emph{smaller than $\nu$ in the convex stochastic order} (denoted as $\mu\lcx\nu$).

The binomial distribution with parameters $n\in\N$ and $p\in[0,1]$ (denoted by $B(n,p)$)
is the probability distribution given by
$$B(n,p)(\{k\})=p_{n,k}(p)=\binom nkp^k(1-p)^{n-k}\quad\text{for }k=0,1,\dots,n$$
and $B(n,p)(\mathbb R\setminus\{0,1,\dots\,n\})=0$.
In particular, $B(1,p)$ is the Bernoulli distribution.

Below we recall the theorem on the binomial convex concentration inequality (see \cite{Shaked2007}).
\begin{theorem}\label{th:Hoeffding}
Let $n\in\N$ , $p_1,\ldots ,p_n \in [0,1]$ and $\overline{p}=\frac{p_1+\dots+p_n}{n}$. Then
$$
B(1,p_1)*\ldots*B(1,p_n) \lcx B(n,\overline{p}).
$$
\end{theorem}

In the above theorem $*$ denotes the convolution of probability distributions.
In \cite{MRW2017}, the authors note that the inequality \eqref{eq:Rasa} is equivalent to the following stochastic convex ordering relation
\begin{equation}\label{eq:prop:2b}
  B(n,x)*  B(n,y) \lcx \frac{1}{2} \left[B(n,x)*  B(n,x) +   B(n,y)*  B(n,y) \right].
\end{equation}

To prove \eqref{eq:prop:2b}, the authors proved the following two propositions on convex ordering relations
\begin{proposition}[\cite{MRW2017}]\label{prop:2a}
\begin{equation}\label{eq:prop:2a}
  B(n,x)*  B(n,y) \lcx B\Bigl(2n,\dfrac{x+y}{2}\Bigr).
	\end{equation}
\end{proposition}
\begin{proposition}[\cite{MRW2017}]\label{prop:2}
 \begin{equation}\label{eq:prop:2}
  B\Bigl(2n,\dfrac{x+y}{2}\Bigr)\lcx \frac{1}{2} \left[B(n,x)*  B(n,x) +   B(n,y)*  B(n,y) \right].
 \end{equation}
\end{proposition}
The inequality \eqref{eq:prop:2a} follows immediately from Theorem \ref{th:Hoeffding} on the binomial convex concentration inequality. In the proof of \eqref{eq:prop:2}, the authors used the Ohlin lemma \cite{Ohl69}.

Ra\c{s}a \cite{Rasa2017b} remarked, that \eqref{eq:Rasa} is equivalent to 
\begin{equation}\label{eq:Rasabis}
 \left(B_{2n}f\right)(x) + \left(B_{2n}f\right)(y)
\geq 2\;\sum_{i=0}^n \sum_{j=0}^n p_{n,i}(x)p_{n,j}(y)\ f\left(\frac{i+j}{2n}\right).
\end{equation}
Since $B_{2n}f$ is convex, we have 
\begin{equation}\label{eq:Rasatres}
 \left(B_{2n}f\right)(x) + \left(B_{2n}f\right)(y)
\geq 2 \left(B_{2n}f\right)\left(\frac{x+y}{2}\right).
\end{equation}
Thus the following problem seems to be a natural one. Prove that 
\begin{equation}\label{eq:Rasa4}
  \left(B_{2n}f\right)\left(\frac{x+y}{2}\right) \geq \sum_{i=0}^n \sum_{j=0}^n p_{n,i}(x)p_{n,j}(y)\ f\left(\frac{i+j}{2n}\right)
\end{equation}
for all convex $f\in\C([0,1])$ and $x,y\in[0,1]$.

If \eqref{eq:Rasa4} is valid, then \eqref{eq:Rasabis} is satisfied,  and hence \eqref{eq:Rasa} is a consequence of \eqref{eq:Rasatres} and \eqref{eq:Rasa4}. Starting from these remarks, Ra\c{s}a \cite{Rasa2017b} presented the inequality \eqref{eq:Rasa4} as an open problem. A very simple probabilistic proof of the inequality \eqref{eq:Rasa4} was given by the authors in the letter to I. Ra\c{s}a
\cite{KomRaj2017}. After that, an analytic proof of \eqref{eq:Rasa4} was given in \cite{AbelRasa2017}.

In this paper, we present the proof of \eqref{eq:Rasa4}, given in the letter to I. Ra\c{s}a \cite{KomRaj2017}, as well as we give generalizations of \eqref{eq:Rasabis}, \eqref{eq:Rasatres}, \eqref{eq:Rasa4} and \eqref{eq:Rasa}. We propose also some generalizations of Theorem \ref{th:Hoeffding} on the binomial convex concentration inequality. Among other, we use it to insert some additional expressions between left and right sides of the Ra\c{s}a inequalities.

\section{Main results}
First we recall a new conjecture of I.\ Ra\c{s}a \cite{Rasa2017b} and present its proof, which we sent in the letter to I. Ra\c{s}a \cite{KomRaj2017}.
\begin{theorem}[new conjecture of I.\ Ra\c{s}a \cite{KomRaj2017}]
\begin{equation}\label{eq:Rasa4bis}
  \sum_{i=0}^n \sum_{j=0}^n p_{n,i}(x)p_{n,j}(y)\ f\left(\frac{i+j}{2n}\right)\leq
	\left(B_{2n}f\right)\left(\frac{x+y}{2}\right) 
\end{equation}
for all convex functions $f\in\C([0,1])$ and $x,y\in[0,1]$.
\end{theorem}
\begin{proof}
Note that \eqref{eq:Rasa4bis} can be written in the form
\begin{equation}
  B(n,x)*  B(n,y) \lcx B\Bigl(2n,\dfrac{x+y}{2}\Bigr),
	\end{equation}
which was proved in \cite{MRW2017} (see Proposition \ref{prop:2a}).
The theorem is proved.
\end{proof}
\par\bigskip

In the following theorem we give a generalization of the inequalities \eqref{eq:Rasa4bis}, \eqref{eq:Rasatres}, \eqref{eq:Rasabis} and \eqref{eq:Rasa}. 
\begin{theorem}
Let $n_i \in \mathbb{N}$ for $i=1,\ldots k$ and $\sum_{i=1}^k n_i=m$. Then
\par\medskip
\begin{equation}\label{eq:6prim}
  \sum_{i_1=0}^{n_1}\ldots \sum_{i_k=0}^{n_k}p_{n_1,i_1}(x_1)\ldots p_{n_k,i_k}(x_k) \ f\left(\frac{i_1+\ldots+i_k}m\right)\leq\left(B_{m}f\right)\left(\sum_{i=1}^k\frac{n_i}m\; x_i\right),
\end{equation}
\begin{equation}\label{eq:7}
  \left(B_{m}f\right)\left(\sum_{i=1}^k\frac{n_i}m\; x_i\right)\leq 
	\sum_{i=1}^k\frac{n_i}m\;\left(B_{m}f\right)\left(x_i\right),
\end{equation}
\par\medskip
\begin{equation}\label{eq:8}
  \sum_{i_1=0}^{n_1}\ldots \sum_{i_k=0}^{n_k}p_{n_1,i_1}(x_1)\ldots p_{n_k,i_k}(x_k) \ f\left(\frac{i_1+\ldots+i_k}m\right)\leq 
	\sum_{i=1}^k\frac{n_i}m\;\left(B_mf\right)\left(x_i\right),
\end{equation}
\par\medskip
\begin{equation}\label{eq:8a}
\sum_{i_1=0}^{n_1}\ldots\sum_{i_k=0}^{n_k} p_{n_1,i_1}(x_1)\ldots p_{n_k,i_k}(x_k) \ f\left(\frac{i_1+\ldots+i_k}m\right)\leq 
	\sum_{i=1}^k\frac{n_i}m\;
	\sum_{j=0}^m p_{m,j}(x_i)f\left(\tfrac jm\right)
\end{equation}
\par\medskip
for all convex functions $f\in\C([0,1])$ and $x_1, \ldots, x_k \in [0,1]$, .
\end{theorem}
\begin{proof}
To prove \eqref{eq:6prim}, using the well-known characterization of binomial distributions, we have that for every  $i=1,\ldots, k$ 
$$
B(n_i,x_i)= \left[B(1,x_i)\right]^{*n_i}=B(1,x_i)*\ldots*B(1,x_i),
$$
which implies
$$
B(n_1,x_1)*\ldots*B(n_k,x_k)=\left[B(1,x_1)\right]^{*n_1}*\ldots*\left[ B(1,x_k)\right]^{*n_k}.
$$
Then by Theorem \ref{th:Hoeffding}, we conclude that
\begin{equation}\label{eq:6a}
 B(n_1,x_1)*\ldots*B(n_k,x_k) \lcx  B(m,\overline{p}),
\end{equation}
where $\overline{p}=\frac{\sum_{i=1}^k n_i\; x_i}{m}=\sum_{i=1}^k\frac{ n_i}{m}\; x_i.$
Since \eqref{eq:6a} is equivalent to \eqref{eq:6prim}, the inequality  \eqref{eq:6prim} is proved.

The inequality \eqref{eq:7} follows immediately from the convexity of $B_{m}f$.
In turn, the inequality \eqref{eq:8} is an immediate consequence of \eqref{eq:6prim} and \eqref{eq:7}, and the inequality \eqref{eq:8a} follows from \eqref{eq:8}.
The theorem is proved.
\end{proof}

\par\bigskip

In the set of all the $m$-tuples $\mathbf p=(p_1,\dots,p_m)\in\mathbb R^m$ we consider the following quasiorder. 
\begin{definition}
We say that $\mathbf q$ majorizes $\mathbf p$ (denoted by $\mathbf p\prec \mathbf q$ or $\mathbf q\succ\mathbf p$) if
\begin{enumerate}[\upshape(i)]
\item $\sum_{l=1}^m\widehat p_l=\sum_{l=1}^m\widehat q_l$,
\item $\sum_{l=1}^k\widehat p_l\leq\sum_{l=1}^k\widehat q_l$ for $k=1,\dots,m$,
\end{enumerate}
where $\widehat p_1\geq\dots\geq\widehat p_m$ and $\widehat q_1\geq\dots\geq\widehat q_m$ are nonincreasing permutations of $\mathbf p$ and $\mathbf q$, respectively.
\end{definition}
The majorization has been studied in \cite{Hardy1952} (before Theorem 45),
\cite{MarshallOlkin2011}, and many other sources.
\par\bigskip

In the next theorem we give a generalization of the binomial convex concentration inequality.
\begin{theorem}\label{th:Hoefbis}
Let $\mathbf p=(p_1,\dots,p_m)\in[0,1]^m$ and $\mathbf p'=(p'_1,\dots,p'_m)\in[0,1]^m$
be such that $\mathbf p$ majorizes $\mathbf p'$ (i.e. $\mathbf p\succ\mathbf p'$).
Then
\begin{equation}\label{eq:Hn}
B(1,p_1)*\dots*B(1,p_m)\leq_{cx}B(1,p'_1)*\dots*B(1,p'_m).
\end{equation}
\end{theorem}
\begin{remark}
Intuitively, Theorem \ref{th:Hoefbis} says that if $\mathbf p'$ is more concentrated than $\mathbf p$ ($\mathbf p\succ\mathbf p'$),
then $B(1,p_1)*\dots*B(1,p_m)$ is more concentrated than $B(1,p'_1)*\dots*B(1,p'_m)$.
\end{remark}
\begin{proof}
Let $\mathbf p\succ\mathbf p'$. We need to show that for each convex function $f:\mathbb R\to\mathbb R$ (or
$f:[0,n]\to\mathbb R$) we have
$$\E f\left(\sum_{i=1}^mX_i\right)\leq \E f\left(\sum_{i=1}^mX'_i\right),$$
where $X_1,\dots,X_m$ and $X'_1,\dots,X'_m$ are independent random variables such that $X_i\sim B(1,p_i)$
and $X'_i\sim B(1,p'_i)$ for each $i=1,\dots,m$.

Since $\mathbf p'$ is majorized by $\mathbf p$, we may fix $\mathbf p^0,\mathbf p^1,\mathbf p^2,\dots,$ $\mathbf p^k\in[0,1]^m$
such that $\mathbf p^0\succ\dots\succ\mathbf p^k$, $\mathbf p^0$ is a permutation of $\mathbf p$,
$\mathbf p^k$ is a permutation of $\mathbf p'$, and such that  
for every $l=1,\dots,k$ there exist $s,t\in\{1,\dots,m\}$ such that
$p^l_i=p^{l-1}_i$ if $i\notin\{s,t\}$,
$p^l_s+p^l_t=p^{l-1}_s+p^{l-1}_t$, and
$p^l_s$ and $p^l_t$ are located between $p^{l-1}_s$ and $p^{l-1}_t$.
In other words $\mathbf p^l$ is constructed from $\mathbf p^{l-1}$ by changing
just two of its terms (making the values of these terms closer).

Because of transitivity of the relation $\leq_{cx}$
it is enough to show \eqref{eq:Hn} for $\mathbf p=\mathbf p^{l-1}$ and $\mathbf p'=\mathbf p^l$, $l=1,\dots,k$. 
Let $s,t\in\{1,\dots,m\}$ be such that $p'_i=p_i$ if $i\notin\{s,t\}$,
$p_s\geq p'_s\geq p'_t\geq p_t$, and $p'_s+p'_t=p_s+p_t$.

Let $X_1,\dots,X_m$ be independent random variables such that $X_i\sim B(1,p_i)$, $i=1,\dots,n$. For $i\notin\{s,t\}$ we define $X'_i=X_i$, and let $X'_s$, $X'_t$ be independent, and independent on $X'_i$, $i\notin\{s,t\}$ and such that $X'_s\sim B(1,p'_s)$, $X'_t\sim B(1,p'_t)$.

Assume that $f:\mathbb R\to\mathbb R$ (or $f:[0,n]\to\mathbb R$)
is a convex function. 
Then we have the equality
\begin{multline*}
\E f\left(\sum_{i=1}^mX'_i\right)=
\E \left((1-p'_s)(1-p'_t)\cdot f\left(\sum_{i\neq s,t}X_i\right)+\right.
\\
\left.((1-p'_s)p'_t+p'_s(1-p'_t))\cdot f\left(1+\sum_{i\neq s,t}X_i\right)
+p'_sp'_t\cdot f\left(2+\sum_{i\neq s,t}X_i\right)\right),\\
\end{multline*}
and similarly for $\E f\left(\sum_{i=1}^mX_i\right)$. It follows that
\begin{multline*}
\E f\left(\sum_{i=1}^mX'_i\right)-\E f\left(\sum_{i=1}^mX_i\right)=
\\
\E \Bigg(((1-p'_s)(1-p'_t)-(1-p_s)(1-p_t))\cdot f\left(\sum_{i\neq s,t}X_i\right)+
((1-p'_s)p'_t+p'_s(1-p'_t)-\\
(1-p_s)p_t+p_s(1-p_t))\cdot f\left(1+\sum_{i\neq s,t}X_i\right)+
(p'_sp'_t-p_sp_t)\cdot f\left(2+\sum_{i\neq s,t}X_i\right)\Bigg)=\\
\E \left(2(p'_sp'_t-p_sp_t)\cdot\left(\frac12f\left(\sum_{i\neq s,t}X_i\right)+\frac12f\left(2+\sum_{i\neq s,t}X_i\right)-f\left(1+\sum_{i\neq s,t}X_i\right)\right)\right)\geq0.
\end{multline*}
The last inequality follows from the fact that 
$$2(p'_sp'_t-p_sp_t)=\frac{(p_s-p_t)^2-(p'_s-p'_t)^2}2\geq0,$$
and the non-negativity of 
$$\frac12f\left(\sum_{i\neq s,t}X_i\right)+\frac12f\left(2+\sum_{i\neq s,t}X_i\right)-f\left(1+\sum_{i\neq s,t}X_i\right)$$
follows from the convexity of the function $f$. The theorem is proved.
\end{proof}

\begin{remark}
Taking in the above theorem $\mathbf p'=(\overline{p},\dots,\overline{p})$, where $\overline{p}=\frac1m\sum_{i=1}^mp_i$,
we obtain the inequality
$B(1,p_1)*\dots*B(1,p_m)\leq_{cx}B(m,\overline{p})$, i.e. the binomial convex concentration inequality given in Theorem \ref{th:Hoeffding}.
\end{remark}

In the following example we show that the condition $\mathbf p\succ\mathbf p'$ in Theorem \ref{th:Hoefbis} is sufficient but it is not necessary.
\begin{example}
Let $\mathbf p=(\frac34,\frac34,0)$ and $\mathbf p'=(\frac56,\frac12,\frac16)$.
We have $p_1+p_2+p_3=p'_1+p'_2+p'_3$ but $\mathbf p\succ\mathbf p'$ is not satisfied
(because $p_1=\max(p_1,p_2,p_3)$ is smaller than $p'_1=\max(p'_1,p'_2,p'_3)$). On the other hand
$B(1,p_1)*B(1,p_2)*B(1,p_3)\leq_{cx}B(1,p'_1)*B(1,p'_2)*B(1,p'_3)$.
Indeed, we have
$\mu:=B(1,p_1)*B(1,p_2)*B(1,p_3)=\frac1{16}\delta_0+\frac38\delta_1+\frac9{16}\delta_2$
and
$\nu:=B(1,p'_1)*B(1,p'_2)*B(1,p'_3)=\frac5{72}\delta_0+\frac{31}{72}\delta_1+\frac{31}{72}\delta_2+\frac5{72}\delta_3$. Then by the Jensen inequality
\begin{multline*}
\int f\ d\nu-\int f\ d\mu=\bigl(\tfrac5{72}f(0)+\tfrac{31}{72}f(1)+\tfrac{31}{72}f(2)+\tfrac5{72}f(3)\bigr)-\\
\bigl(\tfrac1{16}f(0)+\tfrac38f(1)+\tfrac9{16}f(2)\bigr)=
\tfrac{19}{144}\cdot\bigl(\tfrac1{19}f(0)+\tfrac8{19}f(1)+\tfrac{10}{19}f(3)-f(2)\bigr)\geq 0.
\end{multline*}
for each convex functions $f:\mathbb R\to\mathbb R$ (or $f:[0,3]\to\mathbb R$).
\end{example}

The following example shows that it is not true that the condition $\mathbf p\succ\mathbf p'$ can be weakened by replacing it with the conditions $\overline p=\overline{p'}$ and $\sum_{i=1}^m(p_i-\overline p)^2\geq\sum_{i=1}^m(p'_i-\overline{p'})^2$
($\overline p=\frac1m\sum_{i=1}^mp_i$ and $\overline{p'}=\frac1m\sum_{i=1}^mp'_i$), i.e. $\mathbb E(\mathbf p)=\mathbb E(\mathbf p')$ and $Var(\mathbf p)\geq Var(\mathbf p')$.

\begin{example}\rm
Let $\mathbf p=(1,\frac12,\frac12,0)$ and $\mathbf p'=(\frac56,\frac56,\frac16,\frac16)$.
Then we have $\overline p=\frac12=\overline{p'}$ and
$\sum_{i=1}^m(p_i-\overline p)^2=\frac12>\frac49=\sum_{i=1}^m(p'_i-\overline{p'})^2$.
We have also
$\mu:=B(1,p_1)*\dots*B(1,p_4)=\frac14\delta_1+\frac12\delta_2+\frac14\delta_3$
and
$\nu:=B(1,p'_1)*\dots*B(1,p'_4)=\frac{25}{6^4}\delta_0+\frac{260}{6^4}\delta_1+\frac{726}{6^4}\delta_2+\frac{260}{6^4}\delta_3+\frac{25}{6^4}\delta_4$,
which implies that for the convex function $f(x)=|x-2|$ we obtain
$\int f\ d\mu=\frac12>\frac{155}{324}=\int f\ d\nu$, which contradicts the convex ordering relation $\mu\leq_{cx}\nu$.
\end{example}
In the following theorem we give the conditions, which are equivalent to \eqref{eq:Hn}, but this characterization seems completely impractical
(therefore we skip the proof).
\begin{theorem}
Let $\mathbf p=(p_1,\dots,p_m)\in[0,1]^m$ and $\mathbf p'=(p'_1,\dots,p'_m)\in[0,1]^m$.
The following conditions are equivalent:
\begin{enumerate}[\upshape(i)]
\item $B(1,p_1)*\dots*B(1,p_m)\leq_{cx}B(1,p'_1)*\dots*B(1,p'_m)$,
\item $\sigma_1(\mathbf p)=\sigma_1(\mathbf p')$ and
$\forall_{k=2,3,\dots,m}\ \sum_{j=k}^m(-1)^{j-k}\binom{j-2}{k-2}(\sigma_j(\mathbf p')-\sigma_j(\mathbf p))\geq0$,
\end{enumerate}
where $\sigma_1,\dots,\sigma_m$ are symmetric polynomials of  $m$ variables, i.e.
$$\sigma_j(x_1,\dots,x_m)=\sum_{A\subset\{1,\dots,m\},\ |A|=j}\prod_{i\in A}x_i.$$
\end{theorem}
\par\bigskip

In the following theorem we give a generalization of Proposition \ref{prop:2}.
\begin{theorem}
Let $\mathbf p=(p_1,\dots,p_m)\in[0,1]^m$ and $\mathbf p'=(p'_1,\dots,p'_m)\in[0,1]^m$
be such that $\mathbf p\succ\mathbf p'$.
Then 
$$\sum_{i=1}^m(B_nf)(p_i)\geq \sum_{i=1}^m(B_nf)(p'_i)$$
for all convex functions $f:[0,1]\to\mathbb R$.
\end{theorem}

\begin{proof}
Since the function $B_nf:[0,1]\to\mathbb R$ is convex, the theorem follows immediately from the Hardy-Littlewood-P\'olya inequality (\cite{Hardy1952}, Theorem 108).
\end{proof}
\par\bigskip
In the next theorem we use Theorem \ref{th:Hoefbis} and the Jensen inequality to insert some additional expressions between left and right sides of the Ra\c{s}a inequalities \eqref{eq:6prim} and \eqref{eq:7}, respectively.
\begin{theorem}
Let $k \in \mathbb{N}$, $n_i \in \mathbb{N}$ and $x_i \in [0,1]$  for $i=1,\ldots k$. Let 
$$\sum_{i=1}^k n_i=m,\quad\widetilde{n_i}=n_1+\ldots+n_i\quad\text{and}\quad\widetilde{x_i}=\left(\widetilde{n_i}\right)^{-1}(n_1x_1+\ldots+ n_ix_i)\quad\text{for}\quad i=1,\ldots k.$$ 
Then for all convex functions $f\in\C([0,1])$ and $j=3,\ldots k-1$
\begin{multline}\label{eq:200}
  \sum_{i_1=0}^{n_1}\ldots\sum_{i_k=0}^{n_k}p_{n_1,i_1}(x_1)\ldots p_{n_k,i_k}(x_k) \ f\left(\frac{i_1+\ldots+i_k}m\right)\leq\\
	\sum_{i_{j-1}=0}^{\widetilde{n_{j-1}}}\sum_{i_j=0}^{n_j}\ldots\sum_{i_k=0}^{n_k}p_{\widetilde{n_{j-1}},i_{j-1}}(\widetilde{x_{j-1}})p_{n_j,i_j}(x_j)\ldots p_{n_k,i_k}(x_k) \ f\left(\frac{i_{j-1}+\ldots+i_k}m\right)\leq\\
	\sum_{i_j=0}^{\widetilde{n_j}}\sum_{i_{j+1}=0}^{n_{j+1}}\ldots\sum_{i_k=0}^{n_k}p_{\widetilde{n_j},i_j}(\widetilde{x_j})p_{n_{j+1},i_{j+1}}(x_{j+1})\ldots p_{n_k,i_k}(x_k) \ f\left(\frac{i_j+\ldots+i_k}m\right)\leq \ldots \leq
		\left(B_{m}f\right)\left(\sum_{i=1}^k\frac{ n_i}{m}\; x_i\right), 
\end{multline}
\begin{multline}\label{eq:201}
  \left(B_{m}f\right)\left(\sum_{i=1}^k\frac{n_i}m\; x_i\right)\leq\ldots\leq 
	\frac{\widetilde{n_j}}m\;\left(B_m f\right)\left(\widetilde{x_j}\right)+\frac{n_{j+1}}m\;\left(B_m f\right)\left(x_{j+1}\right)+\ldots+\frac{n_k}m\;\left(B_m f\right)\left(x_k\right)\leq\\
	\frac{\widetilde{n_{j-1}}}m\;\left(B_m f\right)\left(\widetilde{x_{j-1}}\right)+\frac{n_j}m\;\left(B_{m}f\right)\left(x_j\right)+\ldots+\frac{n_k}m\;\left(B_m f\right)\left(x_k\right)\leq\ldots\leq
	\sum_{i=1}^k\frac{n_i}m\;\left(B_m f\right)\left(x_i\right).
\end{multline}
\end{theorem}
\begin{proof} Put 
$$
\overline{x}=\widetilde{x_k}=\sum_{i=1}^k\frac{n_i}m\; x_i.
$$
To prove \eqref{eq:200}, we take into account the following equalities 
$$B(n_i,x_i)= \left[ B(1,x_i) \right]^{*n_i}, \quad B(\widetilde{n_i},\widetilde{x_i})= \left[ B(1,\widetilde{x_i}) \right]^{*\widetilde{n_i}}, \quad i=1,\ldots, k.$$
Then by Theorem \ref{th:Hoefbis}, we obtain the following convex ordering relations
\begin{equation*}
B(n_1,x_1)*\ldots*B(n_k,x_k) \lcx B(\widetilde{n_2},\widetilde{x_2})*B(n_3,x_3)*\ldots*B(n_k,x_k) \lcx \ldots \lcx B(m,\overline{x})  ,
\end{equation*}
which are equivalent to the inequalities \eqref{eq:200}.

It is not difficult to prove, that by the convexity of $B_{m}f$, the inequalities \eqref{eq:201} follow immediately from the Jensen inequality. The theorem is proved.
\end{proof}


\end{document}